\documentclass[pdflatex,sn-mathphys-num]{sn-jnl}

\usepackage{graphicx}%
\usepackage{multirow}%
\usepackage{amsmath,amssymb,amsfonts}%
\usepackage{amsthm}%
\usepackage{mathrsfs}%
\usepackage[title]{appendix}%
\usepackage{xcolor}%
\usepackage{textcomp}%
\usepackage{manyfoot}%
\usepackage{booktabs}%
\usepackage{algorithm}%
\usepackage{algorithmicx}%
\usepackage{algpseudocode}%
\usepackage{listings}%

\usepackage{mathtools}
\mathtoolsset{showonlyrefs}

\theoremstyle{thmstyleone}%
\newtheorem{theorem}{Theorem}
\newtheorem{prop}[theorem]{Proposition}%

\theoremstyle{thmstyletwo}%
\newtheorem{lemma}[theorem]{Lemma}
\newtheorem{corollary}[theorem]{Corollary}

\theoremstyle{thmstylethree}%

\def\xl{\underline{x}}
\def\xu{\overline{x}}
\def\cO{\mathcal{O}}
\def\argmin\ensuremath{\operatorname{argmin}}

\begin{document}

\title[Gradient-evaluation sequence in AGD]{A Note on the Gradient-Evaluation Sequence in Accelerated Gradient Methods}


\author[1]{\fnm{Yan} \sur{Wu}}\email{yw8@clemson.edu}

\author[2]{\fnm{Yipeng} \sur{Zhang}}\email{yipengz@clemson.edu}

\author[1]{\fnm{Lu} \sur{Liu}}\email{lliu6@clemson.edu}

\author*[2]{\fnm{Yuyuan} \sur{Ouyang}}\email{yuyuano@clemson.edu}

\affil[1]{\orgdiv{Department of Industrial Engineering}, \orgname{Clemson University}, 
\orgaddress{
\city{Clemson}, \postcode{29634}, \state{South Carolina}, \country{USA}}}

\affil*[2]{\orgdiv{School of Mathematical and Statistical Sciences}, \orgname{Clemson University}, 
\orgaddress{
\city{Clemson}, \postcode{29634}, \state{South Carolina}, \country{USA}}}


\abstract{Nesterov's accelerated gradient descent method (AGD) is a seminal deterministic first-order method known to achieve the optimal order of iteration complexity for solving convex smooth optimization problems. Two distinct sequences of iterates are included in the description of AGD: gradient evaluations are performed at one sequence, while approximate solutions are selected from the other. 
The iteration complexity on minimizing objective function value has been well-studied in the literature, but such analysis is almost always performed only at the approximate solution sequence.
To the best of our knowledge, for projection-based AGD that solves problems with feasible sets, it is still an open research question whether the gradient evaluation sequence (when treated as approximate solutions) could also achieve the same optimal order of iteration complexity. It is also unknown whether such results still hold in the non-Euclidean setting. Motivated by computer-aided algorithm analysis, we provide positive results that answer the open problems affirmatively. Specifically, for (possibly constrained) problem $f^*:=\min_{x\in X}f(x)$ where $f$ is convex and $L$-smooth and $X$ is closed,  convex and projection friendly, we prove that the gradient-evaluation sequence $\{\xl_k\}$ in AGD satisfies that $f(\xl_k) - f^*\le \cO(L/k^2)$. 
}




\maketitle

\section{Introduction}
Since Nesterov's seminal paper \cite{nesterov1983method}, the accelerated gradient method (AGD) has seen a wide range of applications in large-scale optimization. See, e.g., Nesterov's book \cite{nesterov2018lectures} for a thorough description of the method. The idea of using convex combinations of iterations for accelerating the convergence rates of first-order algorithms has been applied to solve many unconstrained and constrained optimization problems; see, e.g., \cite{nesterov2005smooth,chen2017accelerated,ouyang2015accelerated,nesterov2012gradient,nesterov2015universal} and references within. There have also been extensive studies on the fundamental ideas behind AGD, including ordinary differential equation reformulations \cite{su2014differential,shi2022understanding}, game-theoretic perspectives (see \cite{lan2020first}, Section 3.4), geometric interpretations \cite{bubeck2015geometric}, Halpern-iteration-type viewpoints \cite{tran2024halpern}, and asymptotic convergence of iterates in $\mathbb{R}^n$ \cite{jang2025point} and Hilbert spaces \cite{bot2025iterates}. While there exist many literature studying the intuition behind AGD and several variants of AGD, all AGD variants share a similar algorithm design principle through the convex combination of algorithm iterates. Specifically, there are usually two or three sequences of iterations: one that handles gradient evaluation, one that serves as an approximate solution, and one that handles algorithm progression (which might be eliminated from algorithm description for unconstrained problems; see, e.g., Section 2.2.1 in \cite{nesterov2018lectures}). In Algorithm \ref{alg:AGD}, we describe a typical AGD implementation for solving a convex smooth problem
\begin{align}
    \label{eq:problem}
    f^*:=\min_{x\in X}f(x).
\end{align}
Here we assume that an optimal solution $x^*$ to the above problem exists.
Our notations for algorithm description follows the convention in \cite{lan2020first} (see Section 3.3 within). Here for problem \eqref{eq:problem}, the typical assumption is that $f$ is a convex differentiable function and $X$ is a closed convex set. Moreover, for convergence analysis of first-order algorithms, it is usually assumed that $\nabla f$ is $L$-Lipschitz, i.e., 
\begin{align}
    \label{eq:L}
    \|\nabla f(y) - \nabla f(x)\|_*\le L\|y - x\|, \forall x,y\in \mathbb{R}^n.
\end{align}
Here $\|\cdot\|$ is any norm in $\mathbb{R}^n$ and $\|\cdot\|_*$ is its dual norm. In some literature such assumptions are also known as the $L$-smoothness of $f$. When the norm $\|\cdot\|$ is chosen to be the Euclidean norm $\|\cdot\|_2$, we say that the problem is in the Euclidean setting. In such case, the dual norm $\|\cdot\|_*$ is also the Euclidean norm. 

\begin{algorithm}[h]
\caption{\label{alg:AGD} The accelerated gradient descent (AGD) method}
\begin{algorithmic}[1]
\State \textbf{Input:} initial points $\overline x_{0}=x_{0}\in X$, parameters $\gamma_k$ , $\eta_k>0$, for all $k\ge1$, and the total number of iterations $N$.  
\For{$k=1,\ldots,N$}
  \State Compute
  \begin{align}
    \label{eq:xl}
      \xl_k = &(1-\gamma_k)\,\overline x_{k-1} + \gamma_k\, x_{k-1}
      \\
      \label{eq:g}
      g_k = & \nabla f(\xl_k)
      \\
      \label{eq:x}
      x_k = & \operatorname{argmin}_{x\in X}\;
      \big\langle g_k,\, x \big\rangle
      + \eta_k\,V(x_{k-1}, x)
      \\
      \label{eq:xu}
      \overline x_k = & (1-\gamma_k)\,\overline x_{k-1} + \gamma_k\, x_k
  \end{align}
\EndFor
\end{algorithmic}
\end{algorithm}

In the description of Algorithm \ref{alg:AGD}, the Bregman divergence is used to cover both the Euclidean and non-Euclidean settings. 
Specifically, the Bregman divergence $V(x,y)$ is defined as follows. A function $\nu:X\to\mathbb{R}$ is a distance generating function with modulus $\sigma_\nu>0$ with respect to norm  $\|\cdot\|$ if $\nu$ is convex and continuous on $X$, the set
\[
X^{o}=\Bigl\{x\in X:\ \exists\,p\in\mathbb{R}^n \ \text{s.t.}\ x\in\arg\min_{u\in X}\bigl(\langle p,u\rangle+\nu(u)\bigr)\Bigr\}
\]
is convex, and on $X^{o}$ the function $\nu$ is continuously differentiable and $\sigma_\nu$-strongly convex:
\[
\langle \nabla \nu(x)-\nabla \nu(\overline x), x-\overline x\rangle\ge \sigma_\nu\|x-\overline x\|^2,\qquad \forall\,x,\overline x\in X^{o}.
\]
We then define the Bregman (also known as prox-function) $V:X^{o}\times X\to\mathbb{R}_+$ by
\begin{align*}
    V(x,z)=\nu(z)-\Bigl[\nu(x)+\langle\nabla \nu(x),z-x\rangle\Bigr].
\end{align*}
In the Euclidean case with $\nu(x)=\tfrac12\|x\|_2^2$, we have $\sigma_\nu=1$ and $V(x,z)=\tfrac12\|z-x\|_2^2$. In the non-Euclidean case, without loss of generality we can also assume that $\sigma_\nu=1$. The important properties that are commonly used in convergence analysis is that $V(\cdot,\cdot)$ is non-negative and that
\begin{align}
	\label{eq:Vbound}
	V(x,z)\ge \frac{1}{2}\|x-z\|^2.
\end{align}
For details on the use and analysis of Bregman divergences in optimization algorithms, see, e.g., \cite{lan2020first}.

We can clearly observe in the description of Algorithm \ref{alg:AGD} that there are three sequences of iterates with distinct functionality. First, the $\xl_k$ sequence is used for gradient evaluations $\nabla f(\xl_k)$. Second, the $x_k$ sequence serves as the starting point and drives the algorithmic progression in \eqref{eq:x}. Third, the $\xu_k$ sequence is computed as approximate solution outputs: the convergence guarantee of Algorithm \ref{alg:AGD} is that $f(\xu_N) - f(x^*)\le \varepsilon$ after at most $\cO(\sqrt{L/\varepsilon})$ iterations. 
The breakthrough in the seminal paper \cite{nesterov1983method} is that the approximate solution sequence $\{\xu_k\}$ defined in \eqref{eq:xu} satisfies $f(\xu_k) - f(x^*)\le \varepsilon$ when $k\ge \cO(\sqrt{L/\varepsilon})$, a significant improvement upon the $\cO(L/\varepsilon)$ complexity of the gradient descent algorithm. Such an $\cO(\sqrt{L/\varepsilon})$ complexity bound is optimal in order among all first-order algorithms, since there exists worst-case problem instances such that any first-order algorithms require at least $\cO(\sqrt{L/\varepsilon})$ gradient evaluations to compute approximate solutions
(see, e.g., \cite{nesterov2018lectures,nemirovski1983problem,nemirovski1992information} for the discussion on lower and upper complexity bounds).

It should be noted that the $\cO(\sqrt{L/\varepsilon})$ complexity bound of AGD is only optimal in order. An optimized gradient method (OGM) was later developed \cite{kim2016optimized}, whose complexity bound is not only optimal in order, but also optimal in terms of the associated universal constant. Similar optimal constant result for problems with feasible sets is also recently developed in \cite{jang2025computer}. 
It is interesting to observe that in OGM, the sequence for gradient evaluations nearly serves as approximate solutions: the final approximate solution output is expressed in exactly the same form as the gradient-evaluation sequence, differing only in the value of weights. Such observation suggests that, in first-order methods, the iterates used for gradient evaluations may themselves serve as natural candidates for approximate solutions. 

In light of the above observation, one natural question to ask is whether AGD's $\xl_k$ sequence described in Algorithm \ref{alg:AGD} could also serve as approximate solutions. Indeed, such result can be derived when $X=\mathbb{R}^n$; as we will see later in Section \ref{sec: unconstrained optimization convergence analysis}, such derivation is similar to the proof of OGM. However, when $X$ is any closed convex feasible set, it is still an open question whether the $\xl_k$ sequence could serve a similar role. In this paper, we focus on the AGD implementation in the form of Algorithm \ref{alg:AGD} and try to answer the following research question:

\vspace{.5cm}
\emph{For the sequence $\xl_k$ in \eqref{eq:xl}, do we also have the convergence property that $f(\xl_k) - f(x^*) \le \cO(L/k^2)$ for any closed convex feasible set $X$?}
\vspace{.5cm}

In this paper, we provide a positive answer to the above research question. 
Our answer to the above question is mostly motivated by recent developments on computer-aided convergence analysis for first-order methods, which originates from the work on performance estimation problems (PEP) in \cite{drori2014performance}. Briefly speaking, the concept of PEP is to formulate the worst-case analysis of any specified first-order algorithm as an infinite-dimensional optimization problem that finds a function that yields the worst possible performance of such algorithm. By reducing the infinite-dimensional problem to a finite-dimensional one through the convex interpolation theorem, we are then able to analyze the $N$-iteration worst-case performance of such algorithm by solving for function values, gradients, and algorithm iterates that satisfy certain constraints derived from the convex interpolation theorem. Here in order to supply algorithm description to the finite-dimensional relaxation, the common method is to assume that the first-order algorithm of interest satisfies a linear span assumption: $x_k\in x_0 + \text{span}\{\nabla f(x_0),\ldots, \nabla f(x_{k-1})\}$. Such linear span assumption will allow us to further describe algorithm iterates through gradients, yielding a convex optimization problem that involves function values and gradients of the worst-case function. 

However, such linear span assumption does not apply to first-order algorithms for constrained optimization. For example, the AGD algorithm described in Algorithm \ref{alg:AGD} no longer satisfies such assumption due to the projection subproblem in \eqref{eq:x}. Therefore, while one could use PEP to analyze the convergence property of $\underline{x}_k$ iterates of AGD under the assumption that $X=\mathbb{R}^n$ (see Section  \ref{sec: unconstrained optimization convergence analysis}), it is still unclear how to answer the aforementioned research question when $X$ is a general closed convex set. 

To address the issue of PEP for analyzing possibly constrained optimization algorithms, in this paper we start with a naive approach from an equivalent dual perspective of the PEP analysis in \cite{drori2014performance}. Specifically, if we focus on the strategy for solving the finite-dimensional formulation and assuming the Euclidean setting, it is straightforward to observe that the dual problem of such relaxation is exactly to find the best weights for all the inequalities used throughout the analysis. e.g., inequalities of form
\begin{align}
    \label{eq:L_ineq}
    \frac{1}{2L}\|\nabla f(y) - \nabla f(x)\|^2_*\le f(y) - f(x) - \langle\nabla f(x), y - x\rangle \le \frac{L}{2}\|y - x\|^2, \forall x,y\in X.
\end{align}
Therefore, for constrained cases, we may simply treat the optimality conditions of the projection subproblem \eqref{eq:x} as one of such inequalities used for convergence analysis. Under the Euclidean setting when $\|\cdot\|$ is the Euclidean norm $\|\cdot\|_2$ and the Bregman divergence $V(x_{k-1},x) = \tfrac{1}{2}\|x - x_{k-1}\|_2^2$, the optimality conditions are
\begin{align}
    \label{eq:oc_x}
    \langle g_k + \eta_k(x_k - x_{k-1}), x_k - x\rangle \le 0,\ \forall x\in X.
\end{align}
By assigning weights to inequalities of forms \eqref{eq:L_ineq} and \eqref{eq:oc_x} under the Euclidean setting, we are able to construct a PEP for analyzing AGD in the constrained cases. Numerical PEP experiment results could then guide us on developing rigorous theoretical proof on the convergence properties concerning the gradient-evaluation sequence $\{\xl_k\}$.

Two remarks are in place for our theoretical results in this paper. 
First, we emphasize that our intent is not to optimize the constant in the $\cO(1/k^2)$ bound via parameter tuning; this goal has already been achieved by OGM, which no longer matches the classical AGD structure. 
Second, it should be noted that while our results are inspired by PEP, the theoretical developments in this paper do not rely solely on the numerical results of the PEP. Specifically, we use numerical PEP computation results for one representative choice of AGD parameters to identify the pattern of weights for a PEP-type proof, and then generalize such proof to broader parameter settings. Our result in Theorem \ref{thm:main} is also a human-readable one which states that under commonly known AGD assumptions, a generic convergence result holds for $f(\xl_k) - f(x^*)$. We also generalize the result to non-Euclidean settings in Theorem \ref{thm:brg_case}. While PEP is a powerful computer-assisted tool for studying worst-case behavior of algorithms, we believe that our generalization of convergence results to broader parameter and non-Euclidean settings and the development of a human-readable proof in Section \ref{sec: convergence of constrained case} constitute a novel theoretical contribution. Indeed, Section \ref{sec: convergence of constrained case} can be read independently without having any background on PEP. 

This paper is organized as follows. Section~\ref{sec:motivation} 
describes the motivation of our research, which is based on the previously-known convergence property on the approximate solution sequence $\{\xu_k\}$, the convergence property of gradient-evaluation sequence $\{\xl_k\}$ when solving unconstrained problems, and numerical PEP results for general problems. 
Section~\ref{sec: convergence of constrained case} presents our new discoveries on the convergence properties of the gradient-evaluation sequence. Finally, Section~\ref{sec:concluding_remarks} offers some concluding remarks.

\section{Motivation}
\label{sec:motivation}

In this section, we review some previous results on AGD convergence properties and describe the motivation of our research in this paper, including the convergence property of gradient-evaluation sequence in the unconstrained case and PEP numerical results. The convergence analysis concerning sequence $\{\xu_k\}$ in AGD has been well studied. For completeness, we briefly introduce the convergence results below. We define the sequence \(\{\Gamma_k\}\) recursively as:
\begin{align}
    \label{eq:Gamma}
    \Gamma_k = \begin{cases}
        1 & k = 1
        \\
        \Gamma_{k-1}(1-\gamma_k) & k > 1.
    \end{cases}
\end{align}
The convergence properties of AGD concerning sequence $\{\xu_k\}$ is described in the following theorem without a proof. For a reference with proof, see, e.g., \cite{lan2020first}.
\begin{theorem}
    \label{thm:AGD_xu_convergence}
    In Algorithm \ref{alg:AGD}, suppose that the parameters satisfy 
    \begin{align}
    \label{eq:xu_parameters_motivation_section}
    \gamma_1 = 1, \gamma_k \in (0,1),\forall k\ge 2,\text{ and } \eta_k\ge L\gamma_k,\ \forall k\ge 1.
    \end{align}
    Then we have
    \begin{align}
    \label{eq:xu_converge}
        f(\xu_k) - f(x) \le \Gamma_k\sum_{i=1}^{k}\frac{\gamma_i\eta_i}{\Gamma_i}(V(x_{i-1},x) - V(x_i, x)),\ \forall k\ge 1.
    \end{align}
\end{theorem}

There are many parameter settings of AGD that satisfy \eqref{eq:xu_parameters_motivation_section}. Among those parameters, if in addition the sequence $\{\gamma_k\eta_k/\Gamma_k\}$ is non-increasing, then it is straightforward to see from \eqref{eq:xu_converge} that
\begin{align}
    \label{eq:xu_convergence_key}
            f(\xu_k) - f(x) \le \Gamma_k\left({\eta_1}V(x_0,x) - \frac{\gamma_k\eta_k}{\Gamma_k}V(x_k,x)\right),\ \forall x\in X.
\end{align}
Two examples of such parameters are stated in the following corollaries without proof. Both parameters have been mentioned in previous literature studying AGD; see, e.g., \cite{nesterov2018lectures}.
\begin{corollary}
    \label{thm:AGD_par1}
    If the parameters of AGD are set to
    \begin{align}
\begin{aligned}
\label{eq:first parameter choice}
\gamma_k = \frac{2}{k+1}, \eta_k   = \frac{2L}{k}.
\end{aligned}
\end{align}
    then we have
    \begin{align*}
        f(\xu_k) - f(x) \le \frac{4L}{k(k+1)}V(x_0,x),\ \forall k\ge 1.
    \end{align*}
\end{corollary}
\begin{corollary}
    \label{thm:AGD_par2}
    If the parameters of AGD are set to
\begin{align}
\begin{aligned}
\label{eq:second parameter choice}
\gamma_k = \begin{cases}
    1 & k= 1
    \\
    \text{positive solution $\gamma$ to equation $\gamma^2 = \gamma_{k-1}^2(1-\gamma)$} & k>1,
\end{cases}
\ \eta_k = L\gamma_k,
\end{aligned}
\end{align}
    then we have
    \begin{align*}
        f(\xu_k) - f(x) \le \frac{4L}{(k+1)^2}V(x_0,x),\ \forall k\ge 1.
    \end{align*}
\end{corollary}

\vspace{.2cm}

We also have other parameters settings in which $\{\gamma_k\eta_k/\Gamma_k\}$ is non-decreasing. Such parameters are used when we know that the feasible set $X$ is compact. In such case, from \eqref{eq:xu_converge} we have
\begin{align*}
            & f(\xu_k) - f(x) 
            \\
            \le & \Gamma_k\left(\frac{\gamma_1\eta_1}{\Gamma_1}V(x_0,x) - \sum_{i=1}^{k-1}\left(\frac{\gamma_{i}\eta_i}{\Gamma_i} - \frac{\gamma_{i+1}\eta_{i+1}}{\Gamma_{i+1}}\right)V(x_i,x) - \frac{\gamma_k\eta_k}{\Gamma_k}V(x_k,x)\right)
            \\
            \le & \Gamma_k\left(\frac{\gamma_1\eta_1}{\Gamma_1}D_X^2 - \sum_{i=1}^{k-1}\left(\frac{\gamma_{i}\eta_i}{\Gamma_i} - \frac{\gamma_{i+1}\eta_{i+1}}{\Gamma_{i+1}}\right)D_X^2 - \frac{\gamma_k\eta_k}{\Gamma_k}V(x_k,x)\right)
            \\
            \le & {\gamma_k\eta_k}(D_X^2 - V(x_k,x))
            .
\end{align*}
Here we assume that $X$ is bounded and that $D_X:=\sqrt{\max_{x\in X^o,y\in X}V(x,y)}<\infty$.
An example of such parameter setting is described in the following corollary. The proof is easily done by substituting parameter values to the above inequality and is skipped.
\begin{corollary}
    \label{thm:AGD_par3}
    If the feasible set $X$ is compact, $D_X:=\sqrt{\max_{x\in X^o,y\in X}V(x,y)}<\infty$ and 
    the parameters of AGD are set to
    \begin{align}
        \label{eq:third_parameters_choice}
        \gamma_k = \frac{3}{k+2},\ \eta_k = \frac{3L}{k},
    \end{align}
    then we have
    \begin{align*}
        f(\xu_k) - f(x) \le \frac{9LD_X^2}{k(k+2)}.
    \end{align*}
\end{corollary}

Note that the convergence rates of Corollaries \ref{thm:AGD_par1} through \ref{thm:AGD_par3} are all in the order of $\cO(1)L/k^2$ with slightly different $\cO(1)$ constants. However, as discussed previously in the introduction section, it has been known in the literature that none of the aforementioned parameter settings yield the best possible convergence rate with smallest universal constant $\cO(1)$ among all first order methods. To achieve the best possible $\cO(1)$ constant, one needs to modify the AGD algorithm slightly; such modified algorithm is known as the optimized gradient method (OGM), which is developed following the PEP framework \cite{kim2016optimized} (see, also, the survey \cite{d2021acceleration}). Unlike AGD in which there is a specific sequence $\{\xu_k\}$ that is different from the gradient evaluation iterates and used solely for reporting approximate solutions, the approximate solution reported by OGM is closely related to the sequence of gradient evaluation iterates. Observing such difference, one natural research question for AGD is whether the sequence of gradient evaluation iterates could also serve as approximate solution sequence:

\vspace{.5cm}
\emph{For the sequence $\xl_k$ in \eqref{eq:xl}, do we also have the convergence property that $f(\xl_k) - f(x^*) \le \cO(L/k^2)$?} 
\vspace{.5cm}

The unconstrained case when $X=\mathbb{R}^n$ for the above question can be answered relatively easily, while the answer to the case with possibly constrained feasible set $X$ is unknown.
In the remainder of this section, we will first show in Section \ref{sec: unconstrained optimization convergence analysis} that under the unconstrained setting when we use the Euclidean norm {$\|\cdot\|_2$} and $X=\mathbb{R}^n$, by adapting the OGM analysis, the convergence properties of $f(\xl_k) - f(x^*)$ can be derived, but such proof could not be easily adapted to tackle the constrained feasible set case. We will then demonstrate through PEP that numerical evidence on the convergence properties of $f(\xl_k) - f(x^*)$ is available for the constrained feasible set case.
In Subsections \ref{sec: unconstrained optimization convergence analysis} and \ref{sec:computer-aided}, unless specified, we always assume the Euclidean norm {$\|\cdot\|_2$}.
\subsection{Convergence analysis in unconstrained optimization}\label{sec: unconstrained optimization convergence analysis}

In this subsection, we consider a special case when $X=\mathbb{R}^n$ and derive the convergence properties of $f(\xl_k) - f(x^*)$ for one set of AGD parameters. While our analysis is similar to the one for OGM in \cite{kim2016optimized} (see also the analysis in the survey \cite{d2021acceleration}), to the best of our knowledge, such results on AGD have not yet been stated in the literature. 
At the end of this subsection, we will also describe the difficulty of extending such result to a general feasible set $X$.

In our convergence analysis we will make use of the following lemma. The proof can be seen in, e.g., Lemma 2.3 in \cite{lan2022accelerated}.

\begin{lemma}
\label{lem:recursion}
Let $\{\gamma_k\}_{k\ge1}$ satisfy $\gamma_1 = 1$ and 
$\gamma_k \in (0,1)$ for all $k \ge 2$. 
Define the sequence $\{\Gamma_k\}_{k\ge1}$ by \eqref{eq:Gamma}.
Suppose that the sequence $\{a_k\}_{k\ge1}$ satisfies
\begin{align}
    a_k \le (1-\gamma_k)a_{k-1} + b_k, 
    \qquad k \ge 1.
\end{align}
Then, for any integer $N \ge 1$, it holds that
\begin{align}
    a_N \le \Gamma_N \sum_{k=1}^{N} \frac{b_k}{\Gamma_k}.
\end{align}
\end{lemma}

In the following proposition, we describe an important property of the $\{\xl_k\}$ sequence in AGD. For generality, in the description and proof we assume the general (possibly non-Euclidean) setting with any norm $\|\cdot\|$ and the associated dual norm $\|\cdot\|_*$.
\begin{prop}
\label{thm:preprocessing}
Suppose that the parameters $\gamma_k$ in Algorithm \ref{alg:AGD} satisfy $\gamma_1=1$ and $\gamma_k\in (0,1)$ for all $k\ge 2$. For any $N\ge 1$, and any $x\in X$, the iteration $\xl_N$ in Algorithm \ref{alg:AGD} satisfies 
\begin{align}
    \label{eq:preprocessing}
    f(\xl_N) - f(x) \le \Gamma_N\left[\frac{\eta_1}{2}\|x_0 - x\|^2 + \Delta(x)\right],
\end{align}
where
\begin{align*}
    {\Delta}(x)
    &:= \frac{\gamma_N}{\Gamma_N}\langle g_N, x_{N-1} - x_N \rangle
      - \frac{\eta_1}{2}\|x - x_0\|^2 \\
    &\quad + \sum_{k=2}^N \frac{\gamma_{k-1}}{\Gamma_{k-1}}
        \langle g_k - g_{k-1}, x_{k-1} - x_{k-2} \rangle
      - \sum_{k=2}^N \frac{1}{2L\Gamma_{k-1}}\|g_{k-1} - g_k\|_*^2 \\
    &\quad - \sum_{k=1}^N \frac{\gamma_k}{2L\Gamma_k}\|g - g_k\|_*^2
      + \sum_{k=1}^N \frac{\gamma_k}{\Gamma_k}\langle g_k, x_k - x \rangle,
\end{align*}
in which we define $g:=\nabla f(x)$.
\end{prop}

\begin{proof}
    Fix any $k\ge1$. From the convexity and smoothness of $f$ described in \eqref{eq:L_ineq} we have
\begin{align}
\label{eq:tmp1}
\begin{aligned}
    &\, f(\xl_k) - (1 - \gamma_k)f(\xl_{k-1}) - \gamma_k f(x) 
    \\
    =\,& (1 - \gamma_k)(f(\xl_k) - f(\xl_{k-1})) + \gamma_k (f(\xl_k) - f(x) ) 
    \\
    \le\,&-(1 - \gamma_k)\left(\langle g_k, \xl_{k-1} - \xl_k \rangle
        + \frac{1}{2L}\|g_{k-1} - g_k\|^2_*\right) 
    \\
    & 
    - \gamma_k\left( \langle g_k, x - \xl_k \rangle
        + \frac{1}{2L}\|g - g_k\|^2_*\right)
        \\
        =&\,  \langle g_k, \xl_k - (1 - \gamma_k)\xl_{k-1} - \gamma_k x\rangle
        - \frac{1 - \gamma_k}{2L}\| g_{k-1} - g_k\|^2_* - \frac{\gamma_k}{2L}\|g - g_k\|^2_*.
\end{aligned}
\end{align}
Here, from the descriptions of $\xu_k$ and $\xl_k$ in \eqref{eq:xu} and \eqref{eq:xl}, it follows that 
\begin{align}
    \xl_k - (1 - \gamma_k)\xl_{k-1} - \gamma_k x 
    =& 
    (1-\gamma_k)\overline{x}_{k-1} + \gamma_kx_{k-1} - (1 - \gamma_k)\xl_{k-1} - \gamma_k x 
    \\
    =&
 (1-\gamma_k)\gamma_{k-1}(x_{k-1}-x_{k-2})+   \gamma_kx_{k-1}  -\gamma_kx.
\end{align} Consequently, we observe that the inner product term can be rewritten as 
\begin{align}\label{eq:rev_1}
\begin{aligned}    
    & \langle g_k, \xl_k - (1 - \gamma_k)\xl_{k-1} - \gamma_k x \rangle 
    \\
    = & (1-\gamma_k)\gamma_{k-1}\langle g_k, x_{k-1} - x_{k-2}\rangle + \gamma_k\langle g_k, x_{k-1}-x\rangle
    \\
    =&
    (1 - \gamma_k)\gamma_{k-1}\langle g_{k-1}, x_{k-1} - x_{k-2} \rangle + (1 - \gamma_k)\gamma_{k-1}\langle g_{k} - g_{k-1}, x_{k-1} - x_{k-2} \rangle 
    \\
    &+ 
    \gamma_k \langle g_k, x_{k-1} - x_k \rangle
       + \gamma_k \langle g_k, x_k - x \rangle.
\end{aligned}
\end{align}
In the above expression, for convenience in handling the $k=1$ case we assume that $\gamma_0 = 1$, $x_{-1} = x_0$, and $g_0 = \nabla f(x_0)$. It should be pointed out that all such assumptions simply serve as  notational placeholder; when $k=1$ we have $1-\gamma_1=0$ and hence all terms with notations $\gamma_0$, $x_{-1}$, and $g_0$ in equation \eqref{eq:rev_1} vanishes.
Applying the above observation described in \eqref{eq:rev_1} to \eqref{eq:tmp1} and rearranging terms, we have
\begin{align}
    \begin{aligned}
    &f(\xl_k) - f(x) - \gamma_k \langle g_k, x_{k-1} - x_k \rangle\\
    \le\, & (1 - \gamma_k)\big(f(\xl_{k-1}) - f(x) - \gamma_{k-1}\langle g_{k-1}, x_{k-2} - x_{k-1}\rangle\big)\\
    & + (1 - \gamma_k)\gamma_{k-1}\langle g_k - g_{k-1}, x_{k-1} - x_{k-2}\rangle + \gamma_k \langle g_k, x_k -x\rangle\\
    & - \frac{1 - \gamma_k}{2L}\|g_{k-1} - g_k\|^2_*
      - \frac{\gamma_k}{2L}\|g - g_k\|^2_*. 
    \end{aligned}\label{eq:mainrecurrence}
\end{align}

From the above relation, applying Lemma \ref{lem:recursion} and observing that $\gamma_1=1$, we have
\begin{align}
\begin{aligned}
    &\,f(\xl_N) - f(x) - \gamma_N\langle g_N, x_{N-1} - x_N \rangle
    \\
    \le &\, \Gamma_N\left[\sum_{k=2}^N \frac{(1 - \gamma_k)\gamma_{k-1}}{\Gamma_k}
        \langle g_k - g_{k-1}, x_{k-1} - x_{k-2} \rangle - \sum_{k=2}^N \frac{1 - \gamma_k}{2L\Gamma_k}\|g_{k-1} - g_k\|^2_*  \right.
     \\
     &\ \left.- \sum_{k=1}^N \frac{\gamma_k}{2L\Gamma_k}\|g - g_k\|^2_*
      + \sum_{k=1}^N \frac{\gamma_k}{\Gamma_k}\langle g_k, x_k - x \rangle\right].
\end{aligned}
\end{align}
Rearranging terms and noting the recursive definition of $\Gamma_k$ in \eqref{eq:Gamma}, we can observe that the above is equivalent to \eqref{eq:preprocessing}.

\end{proof}

In the above result, we show that the error on objective function value $f(\xl_N) - f(x)$ is bounded by two terms: the distance term $(\eta_1/2)\|x_0 - x\|^2$ and an error term $\Delta(x)$. Next, we will estimate a bound of the error term $\Delta(x)$ under a specific choice of parameter setting previously described in Corollary \ref{thm:AGD_par1}. Here and throughout the end of this section, we will assume the Euclidean setting and always use the Euclidean norm $\|\cdot\|_2$.

\begin{theorem}\label{lem:scaled}
When applying Algorithm \ref{alg:AGD} to optimization problem with unconstrained and Euclidean setting (i.e., $X=\mathbb{R}^n$ and we use {$\|\cdot\|_2$}), if the parameters are set to $\gamma_k=\frac{2}{k+1}$ and $\eta_k=\frac{2L}{k}$, then we have
\[
f(\xl_N)-f(x^*)
\;\le\; \frac{2L}{N(N+1)}\,{\|x^*-x_0\|_2^2},
\]
where $x^*$ is an optimal solution to the problem $\min_{x\in\mathbb{R}^n}f(x)$ .
\end{theorem}

\begin{proof}
When $X=\mathbb{R}^n$ and we use the Euclidean norm {$\|\cdot\|_2$}, the optimization problem in \eqref{eq:x} reduces to $x_k = \operatorname{argmin}_{x\in \mathbb{R}^n} \langle g_k, x\rangle +\frac{\eta_k}{2}{\|x-x_{k-1}\|_2^2}$. By the first-order necessary and sufficient condition of the above problem, we obtain
$g_k=\eta_k(x_{k-1}-x_k)
$.
Hence, for any $x\in\mathbb{R}^n$, and $k\ge1$, we have
\begin{align*}
    \langle g_k, x_k-x\rangle = \eta_k\langle x_{k-1}-x_k,x_k-x\rangle = \frac{\eta_k}{2}\left({\|x-x_{k-1}\|_2^2} - {\|x-x_k\|_2^2} - {\|x_{k-1}-x_{k}\|_2^2}\right).
\end{align*}
Applying the above equality to the conclusion of Proposition \ref{thm:preprocessing}, we obtain

\begin{align*}
    &f(\xl_N) - f(x) + \gamma_N\,\langle g_N,\,x_N - x_{N-1} \rangle\\
    \le \,& \Gamma_N \left[ \sum_{k=1}^N \frac{\gamma_k\eta_k}{2\Gamma_k}({\|x-x_{k-1}\|_2^2} - {\|x-x_{k}\|_2^2}) \right.\\
    & - \sum_{k=1}^N  \frac{\gamma_k\eta_k}{2\Gamma_k}{\|x_{k-1}-x_k\|_2^2} + \sum_{k=2}^N \frac{\gamma_{k-1}}{\Gamma_{k-1}}\,\langle g_k - g_{k-1},\,x_{k-1}-x_{k-2}\rangle \\
    & \left. - \sum_{k=2}^N \frac{1}{2\Gamma_{k-1}L}{\|g_k-g_{k-1}\|_2^2} - \sum_{k=1}^N \frac{\gamma_k}{2\Gamma_kL}{\|g_k-g\|_2^2} \right].
\end{align*}
Given the parameter choice $\gamma_k = 2/(k+1)$ and the definition of $\Gamma_k$ in \eqref{eq:Gamma}, we obtain that
$\Gamma_k = 2/[k(k+1)]$ for all $k \ge 1$. Substituting the expressions of $\gamma_k$, $\Gamma_k$, and $\eta_k$, setting $x = x^*$, multiplying both sides by $2/\Gamma_N$, and rearranging the terms we conclude that
\begin{align}
    \label{eq:unc_issue}
\begin{aligned}
    & N(N+1)\left(f(\xl_N) - f(x^*)\right)\\
    \le \, & 2L{\|x^*-x_0\|_2^2} - 2L{\|x^* - x_N\|_2^2} + 2N\,\langle g_N,\,x_{N-1} - x_N \rangle \\
    & + \sum_{k=1}^N \left[\vphantom{\frac{1}{2}}2(k-1)\langle g_k -g_{k-1}, \, x_{k-1} - x_{k-2}\rangle - 2L{\|x_{k-1} - x_k\|_2^2}
    \right.\\
    &\left. - \frac{(k-1)k}{2L}{\|g_{k-1} - g_k\|_2^2} - \frac{k}{L}{\|g^* - g_k\|_2^2}\right] \\
    = \, & 2L{\|x^*-x_0\|_2^2} - 2L{\|x^* - x_N\|_2^2} + \frac{N^2}{2L}{\|g_N\|_2^2}\\
    & + \sum_{k=2}^N\left[-\frac{(k-1)k}{2L}{\|g_k\|_2^2} + \frac{k-1}{L}\langle g_{k-1}, \, g_k\rangle - \frac{k-1}{2L}{\|g_{k-1}\|_2^2}\right] - \sum_{k=1}^N \frac{k}{L}{\|g_k\|_2^2}
    \\
    \le \, & 2L{\|x^*-x_0\|_2^2}  + \frac{N^2}{2L}{\|g_N\|_2^2}- \sum_{k=2}^N\frac{(k-1)^2}{2L}{\|g_k\|_2^2} - \frac{N}{L}{\|g_N\|_2^2}
    \\
    \le \, & 2L{\|x^*-x_0\|_2^2}
\end{aligned}
\end{align}
Here the equality follows from \(x_k = x_{k-1} - (k/2L)g_k\) and \(g^* = 0\).
\end{proof}

It should be noted that two important equalities are needed in order to finish the proof of Theorem \ref{lem:scaled}, especially in the first equality of \eqref{eq:unc_issue}. Specifically, we use the fact that 
\(x_k = x_{k-1} - (k/2L)g_k\) and \(g^* = 0\), where the first comes from the optimality condition of $x_k$ in subproblem \eqref{eq:x} when $X=\mathbb{R}^n$, and the second is the optimality condition of $x^*$ for problem $\min_{x\in\mathbb{R}^n}f(x)$. 
The proof of Theorem \ref{lem:scaled} will no longer be valid if we replace the two optimality condition equalities to their respective inequality versions $\sup_{x\in X}\langle g_k + \eta_k(x_k - x_{k-1}), x_k-x\rangle\le 0$ and $\sup_{x\in X}\langle g^*, x^* -x\rangle\le 0$ for general $X$. It will also not be valid under the non-Euclidean setting, since \(x_k = x_{k-1} - (k/2L)g_k\) is no longer true.

In order to study the proposed research problem for general $X$, we believe it is the best to trace back the fundamental ideas of the proof in Theorem \ref{lem:scaled}. Specifically, the proof is an adaptation of that of OGM, which is motivated by the PEP framework. Therefore, it is natural to study the PEP framework for analyzing AGD when $X$ is a general feasible set. For such case, some modifications of the original PEP formulation is necessary. We describe our PEP framework modifications and numerical results in the following subsection.

\subsection{Computer-aided analysis}
\label{sec:computer-aided}
As mentioned in the introduction, it is not straightforward to apply the PEP analysis framework for convergence analysis on computing approximate solutions to problems involving a feasible set. The main reason is that the original PEP framework relies on the assumption that any new approximate solution is generated from the linear span of all previous gradients. Such assumption no longer holds for algorithms that solves constrained optimization problems. However, from the dual perspective of the PEP framework, we may understand it in an alternative way, namely, to select appropriate nonnegative weights on inequalities used in the convergence analysis, and compute the best weights that yield the convergence results directly. Specifically, for problem \eqref{eq:problem}, under the Euclidean setting with norm $\|\cdot\|_2$, the inequalities of interest are of the following forms:
\begin{enumerate}
    \item The inequalities (both ends) defined in \eqref{eq:L_ineq}, in which $x$ and $y$ are selected from 
    \begin{align}
        \label{eq:V}
        \mathcal V:=\{\xu_0, \ldots, \xu_N, \xl_1, \ldots, \xl_N, x_0, \ldots, x_N, x^*\}.
    \end{align}
    \item The optimality conditions of $x_k$'s in solving the subproblem \eqref{eq:x}, $k=1,\ldots,N$:
    \begin{align}
        \label{eq:xk_oc}
        \langle g_k + \eta_k(x_k - x_{k-1}), x_k - x\rangle \le 0,\ \forall x\in X.
    \end{align}
    While the above inequality holds for all $x\in X$, we will only apply it at all the points involved in the algorithm, i.e., $x\in \ensuremath{\mathcal V}$ where $\ensuremath{\mathcal V}$ is defined in \eqref{eq:V}. Moreover, noting that the left-hand-side of the inequality \eqref{eq:xk_oc} is linear with respect to $x$, we only need to focus on all $x$ in a subset $\ensuremath{\mathcal W}$ of $\ensuremath{\mathcal V}$, where
    \begin{align}
        \label{eq:W}
        \ensuremath{\mathcal W}:=\{x_0,\ldots, x_N, x^*\}.
    \end{align}
    This is since all the inequalities of form \eqref{eq:xk_oc} at points in $\ensuremath{\mathcal V}$ can be generated from the linear combinations of that in $\ensuremath{\mathcal W}$.
    \item Distance to optimal solution: {$\|x_0 - x^*\|_2\le R$}. This is only a ``placeholder'' inequality. The sole use of it is to make sure that in our convergence result, the objective function value error $f(x) - f(x^*)$ is bounded by $\mathcal{O}(LR^2/N^2)$. 
\end{enumerate}
Summarizing the above inequalities, by applying the PEP framework we are looking for appropriate nonnegative weights $a_{v,v'}, b_{v,v'}, c_{k,w}, d$, where $v,v'\in \ensuremath{\mathcal V}$, $w\in \ensuremath{\mathcal W}$, and $k=1,\ldots,N$, such that a convergence property of form 
\begin{align}
    f(\xl_N) - f(x^*) \le d R^2
\end{align}
can be derived from the following sets of inequalities multiplied by our weights:
\begin{align}
    & a_{v,v'} \left(\frac{1}{2L}{\|\nabla f(v) - \nabla f(v')\|_2^2} - (f(v) - f(v') - \langle\nabla f(v'), v - v'\rangle)\right) \le 0,\ \forall v,v'\in \ensuremath{\mathcal V}
    \\
    & b_{v,v'} \left(f(v) - f(v') - \langle\nabla f(v'), v - v'\rangle - \frac{L}{2}{\|v - v'\|_2^2}\right)\le 0,\ \forall v,v'\in \ensuremath{\mathcal V}
    \\
    & c_{k,w}\langle\nabla f(\xl_k) + \eta_k(x_k - x_{k-1}), x_k - w\rangle \le 0,\ \forall w\in \ensuremath{\mathcal W}, \ k=1,\ldots,N
    \\
    & d({\|x_0 - x^*\|_2^2} - R^2)\le 0.
\end{align}
The PEP optimization problem for computing the proper weights is now the following: find nonnegative $a_{v,v'}, b_{v,v'}, c_{k,w}, d$, such that $d$ is minimized, subject to the following constraints:
\begin{align}
    \label{eq:PEPcons_linear}
    & \sum_{v, v'\in \ensuremath{\mathcal V}}a_{v,v'} \left( - f(v) + f(v')\right) +  \sum_{v, v'\in \ensuremath{\mathcal V}}b_{v,v'} \left(f(v) - f(v')\right) = f(\xl_N) - f(x^*) 
\end{align}    
and    
\begin{align}
    \label{eq:PEPcons_sdp}
    & \begin{aligned}
    & \sum_{v, v'\in \ensuremath{\mathcal V}}a_{v,v'} \left(\frac{1}{2L}{\|\nabla f(v) - \nabla f(v')\|_2^2} - ( - \langle\nabla f(v'), v - v'\rangle)\right) 
    \\
    & + \sum_{v, v'\in \ensuremath{\mathcal V}} b_{v,v'} \left(- \langle\nabla f(v'), v - v'\rangle - \frac{L}{2}{\|v - v'\|_2^2}\right)
    \\
    & + \sum_{w\in \ensuremath{\mathcal W}}\sum_{k=1}^N c_{k,w}\langle\nabla f(\xl_k) + \eta_k(x_k - x_{k-1}), x_k - w\rangle  + d({\|x_0 - x^*\|_2^2})
    \\
    & \ge 0.
    \end{aligned}
\end{align}
It should be noted that relation \eqref{eq:PEPcons_linear} is satisfied as long as all $a_{v,v'}$ and $b_{v,v'}$'s satisfy a series of linear constraints: after summing up, the coefficients of $f(\xl_N)$ and $f(x^*)$ should be $1$ and $-1$ respectively, and the coefficients of all other $f(x)$ with $x\in \ensuremath{\mathcal V}$ should all be $0$. Moreover, note that relation \eqref{eq:PEPcons_sdp} is satisfied as long as all the weights satisfy a semidefinite constraint. This is since the left-hand-size of \eqref{eq:PEPcons_sdp} is a quadratic form with respect to elements in $\{\nabla f(v)\}_{v\in \ensuremath{\mathcal V}}$ and $\ensuremath{\mathcal V}$. Furthermore, noting that $\xu_0, \ldots, \xu_N, \xl_0,\ldots, \xl_N\in \ensuremath{\mathcal V}$ are all convex combinations of $x_0, \ldots, x_N$, Such quadratic form can also be understood as one with respect to elements in $\{\nabla f(v)\}_{v\in \ensuremath{\mathcal V}}$ and $\ensuremath{\mathcal W}$.
Such quadratic form is nonnegative everywhere if and only if its hessian is positive semidefinite. 
Summarizing the above discussions, we conclude with the following semidefinite program (SDP):
\begin{align}
    \label{eq:PEPopt_general}
    \begin{aligned}
    d_N^*:=\min_{a,b,c,d}\ d
    \\
    \text{s.t}\ & 
    \begin{aligned}[t]
        & a_{v,v'}, b_{v,v'}, c_{k,w}, d \ge 0, \ \forall v,v'\in \ensuremath{\mathcal V}, w\in \ensuremath{\mathcal W}, k=1,\ldots,N
        \\
        & \text{$a,b,c,d$ satisfy the linear and semidefinite constraints induced by 
        }
        \\
        & \text{\eqref{eq:PEPcons_linear} and \eqref{eq:PEPcons_sdp} (see the discussion after \eqref{eq:PEPcons_sdp} )}
    \end{aligned}
    \end{aligned}
\end{align}
The SDP incorporates all the information for convergence analysis through the dual perspective of the PEP framework, including the algorithm description and all the inequalities that will be needed in the analysis. Note that for different choices of algorithm parameters (e.g., the ones described in Corollaries \ref{thm:AGD_par1}--\ref{thm:AGD_par3}), the convex combination weights of $\xu_0, \ldots, \xu_N, \xl_0,\ldots, \xl_N\in \ensuremath{\mathcal V}$ with respect to $x_0, \ldots, x_N$ are different, yielding different positive semidefinite constraints.

In Figure \ref{fig:general_xl_vs_N2}, we illustrate the computed optimal value of the SDP \eqref{eq:PEPopt_general} when the parameters in the AGD algorithm follows Corollary \ref{thm:AGD_par1}. For simplicity and without loss of generality, we let the Lipschitz constant $L=1$.
\begin{figure}[h]
    \centering
    \includegraphics[width=0.5\linewidth, trim=50 150 50 150, clip]{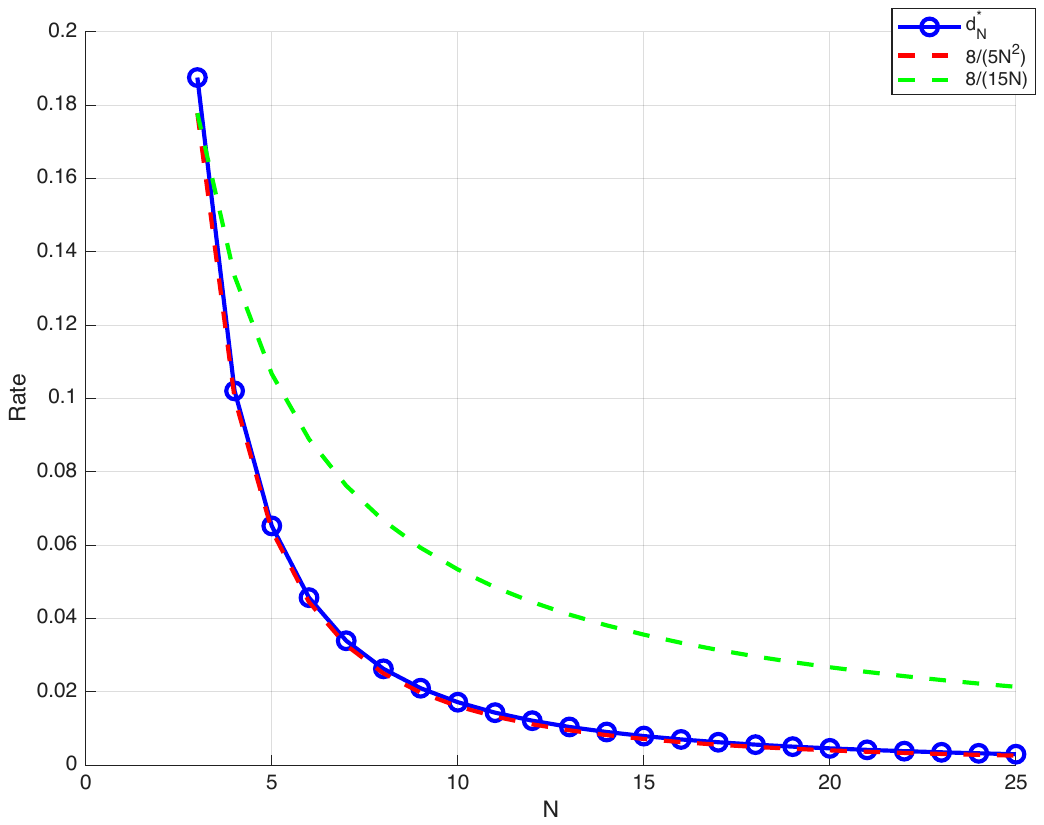}
    \caption{\label{fig:general_xl_vs_N2} Convergence rate result $d_N^*$ among different choices of maximum number of iterations $N$. Here as reference we also draw the curves $N\mapsto 8/(15N)$ and $N\mapsto 8/(5N^2)$ (different constants are chosen so that the two reference curves coincide at the start when $N=3$) so we can visualize the rate of convergence of $d_N^*$.}
\end{figure}
In the figure, the blue curve illustrates the dependence of the numerically computed optimal value $d_N^*$ with respect to the maximum number of iterations $N$. As references, we also draw the curves of functions $8/(5N^2)$ (in red dashed line) and $8/(15N)$ (in green dashed line). The figure clearly demonstrates that PEP derived convergence result $f(\xl_N) - f(x^*)\le d_N^*R^2$ in which $d_N^*$ is approximately in the order $\mathcal{O}(1/N^2)$ and is much better than the order $\mathcal{O}(1/N)$.

Although the result in Figure \ref{fig:general_xl_vs_N2} demonstrates the possible existence of an $\mathcal{O}(1/N^2)$ convergence rate for the sequence $\xl_k$, 
our numerical result does not reveal a systematic construction of the required weights of $a_{v,v'}$, $b_{v,v'}$, $c_{k,w}$, and $d$ to prove the convergence property for any number of iterations $N$. 
Luckily, with the proof of the unconstrained case in the previous section, we can draw some insights in the values of some weights. 
First, we observed that not all inequalities are utilized in the proof of the unconstrained case. Specifically, one side of the inequalities $f(v) - f(v') - \langle\nabla f(v'), v - v'\rangle \le \frac{L}{2}{\|v - v'\|_2^2}$ is never needed in such proof for unconstrained case. Therefore, it is reasonable for us to drop these inequalities and set their associated weights $b_{v,v'}=0$ in the SDP \eqref{eq:PEPopt_general}. 
Second, in the unconstrained proof, the points $v$ and $v'$ used for inequalities $\frac{1}{2L}{\|\nabla f(v) - \nabla f(v')\|_2^2}\le f(v) - f(v') - \langle\nabla f(v'), v - v'\rangle$ are selected only from set $\ensuremath{\mathcal V}':=\{ \xl_0, \ldots, \xl_N, x^*\}$. This suggests that for the constrained case, we might also attempt to only consider the relationship between $v$ and $v'$ in $\ensuremath{\mathcal V}'$, i.e., $a_{v,v'}=0$ unless $v,v'\in \ensuremath{\mathcal V}'$. 
Third, since the unconstrained proof provides an explicit form for the weights of inequalities $\frac{1}{2L}{\|\nabla f(v) - \nabla f(v')\|_2^2}\le f(v) - f(v') - \langle\nabla f(v'), v - v'\rangle$, it is reasonable for us to adopt the same form:
\begin{align*}
    a_{\xl_{k},\xl_{k+1}}=\frac{k(k-1)}{N(N+1)},\ a_{\xl_k,x^*}=\frac{2(k-1)}{N(N+1)}, \quad \forall k \in {1, \ldots, N}.
\end{align*}
Fourth, in the unconstrained proof, for inequalities $\langle \nabla f(\xl_k) + \eta_k(x_k - x_{k-1}), x_k - x\rangle \le 0$ we only use the result for $x=x^*$. 
 Based on this observation, we could attempt by assuming that $c_{k,w}=0$ for any $w \neq x^*$. 
 
Fixing the values of the weights based on the above observations and analysis, the only remaining weights to be optimized are $c_{k,x^*}$'s, $k=1,\ldots, N$. 
Note that by fixing some weights we are no longer computing the optimal value $d_N^*$ in the SDP \eqref{eq:PEPopt_general}. We denote the optimal value of \eqref{eq:PEPopt_general} with aforementioned fixed weights as $d_N^{*'}$. 
We conduct numerical experiments to compute $d_N^{*'}$ with maximum number of iteration $N=3,\ldots,25$. 
In Figure \ref{fig:general_xl_vs_N2_fix} we illustrate the dependence of $d_N^{*'}$ on $N$, along with reference curves for $11/(5N^2)$ (red dashed line) and $11/(15N)$ (green dashed line). We also list all the computed numerical values of $c_{k,x^*}$'s in Table \ref{tab:Wocstar results}.

\begin{figure}[!h]
    \centering
    \includegraphics[width=0.5\linewidth, trim=50 150 50 150, clip]{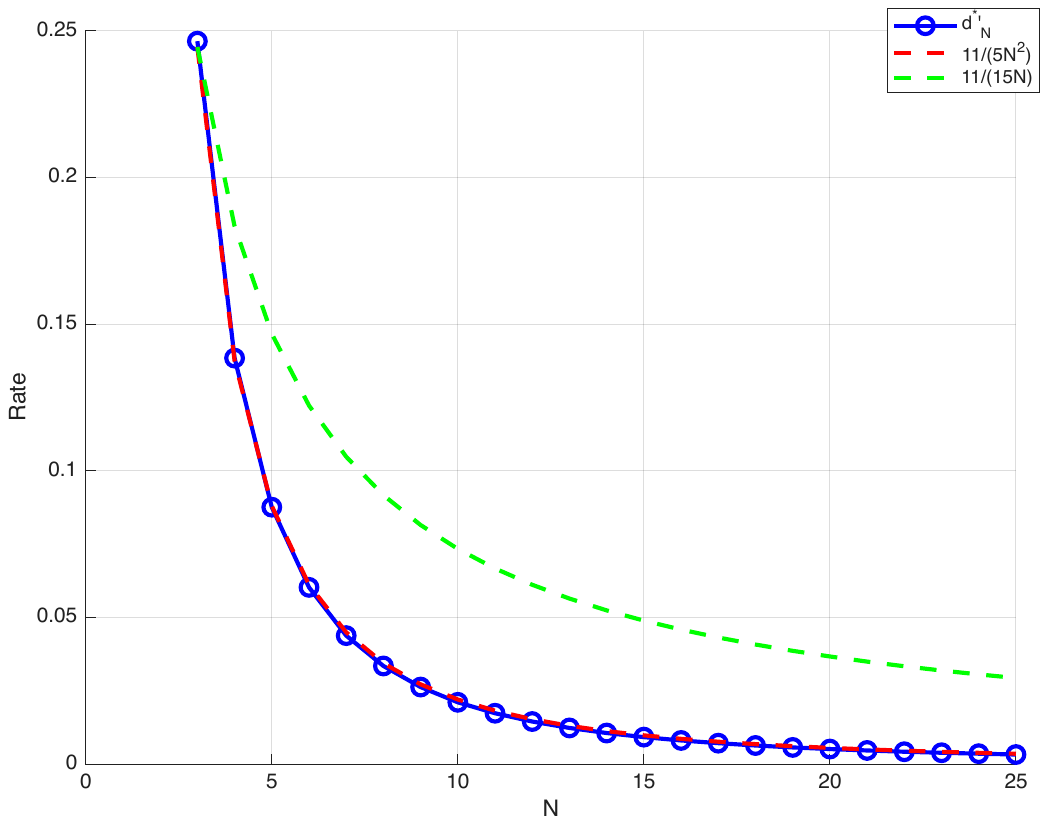}
    \caption{\label{fig:general_xl_vs_N2_fix} Convergence rate result $d_N^{*'}$ among different choices of maximum number of iterations $N$. Here as reference we also draw the curves $N\mapsto 11/(15N)$ and $N\mapsto 11/(5N^2)$ (different constants are chosen so that the two reference curves coincide at the start when $N=3$) so we can visualize the rate of convergence of $d_N^{*'}$.}
\end{figure}
\begin{table}[!htbp]
    \centering
    \resizebox{\columnwidth}{!}{\begin{tabular}{c|c|c|c|c|c|c|c|c|c|c|c|c|c}
    \toprule
         &  $N=3$&$N=4$&$N=5$&$N=6$&$N=7$&$N=8$&$N=9$&$N=10$&$N=11$&$N=12$&$N=13$&$N=14$&$N=15$\\
        \hline
         $k$& $c_{k,x^*}$& $c_{k,x^*}$& $c_{k,x^*}$& $c_{k,x^*}$& $c_{k,x^*}$& $c_{k,x^*}$& $c_{k,x^*}$& $c_{k,x^*}$& $c_{k,x^*}$& $c_{k,x^*}$& $c_{k,x^*}$& $c_{k,x^*}$& $c_{k,x^*}$  \\
        \hline
        1&2.0067&2.0040&2.0076&2.0094&2.0127&2.0090&2.0101&2.0173&2.0178&2.0203&2.0214&2.0244&2.0258\\
        \hline
        2&9.9889&4.0061&3.9992&3.9982&3.9972&3.9974&3.9968&3.9952&3.9953&3.9943&3.9941&3.9935&3.9940\\
        \hline
        3&0.0000&13.9847&6.0121&6.0006&5.9997&5.9994&5.9990&5.9978&5.9976&5.9971&5.9967&5.9959&5.9953\\
        \hline
        4&&0.0000&17.9724&8.00143&8.0015&8.0003&7.9999&7.9994&7.9991&7.9987&7.9984&7.9979&7.9976\\
        \hline
        5&&&0.0000&21.9679&10.0196&10.0011&10.0006&10.0004&9.9999&9.9996&9.9993&9.9990&9.9987\\
        \hline
        6&&&&0.0000&25.9563&12.0123&12.0013&12.0014&12.0007&12.0003&12.0000&11.9996&11.9994\\
        \hline
        7&&&&&0.0000&29.9726&14.0132&14.0026&14.0016&14.0010&14.0006&14.0002&14.0000\\
        \hline
        8&&&&&&0.0000&33.9705&16.0243&16.0027&16.0019&16.0013&16.0008&16.0005\\
        \hline
        9&&&&&&&0.0000&37.9454&18.0242&18.0030&28.0021&28.0015&18.0011\\
        \hline
        10&&&&&&&&0.00000&41.9447&20.0263&20.0033&20.0024&20.0018\\
        \hline
        11&&&&&&&&&0.0000&45.9397&22.0275&22.0036&22.0027\\
        \hline
        12&&&&&&&&&&0.0000&49.9361&24.0297&24.0040\\
        \hline
        13&&&&&&&&&&&0.0000&53.9308&26.0328\\
        \hline
        14&&&&&&&&&&&&0.0000&57.9228\\
        \hline
        15&&&&&&&&&&&&&0.0000\\
    \bottomrule
    \end{tabular}}
    \caption{Results of $c_{k,x^*}N(N+1)$ for different choices of maximum number of iterations $N$}
    \label{tab:Wocstar results}
\end{table}

Figure \ref{fig:general_xl_vs_N2_fix} confirms that the PEP convergence result $f(\xl_N) - f(x^*)\le d_N^{*'}R^2$ remains approximately $\mathcal{O}(1/N^2)$ and is much better than the order $\mathcal{O}(1/N)$. 
Moreover, the numerical results of $c_{k,x^*}$ for different $N$ suggest that a good choice of weights $c_{k,x^*}$ follows the pattern

\begin{align}
    \label{eq:c_wk}
    c_{k,x^*} = \begin{cases}
        \frac{2k}{N(N+1)} & k < N-2
        \\
        \frac{4N-2}{N(N+1)} & k = N-1
        \\
        0 & k=N.
    \end{cases}
\end{align}
Through this computer-aided analysis, we have successfully obtained all necessary nonnegative weights via numerical solutions from the PEP framework.

It should be noted that the above computer-aided analysis only provides a conjecture concerning the convergence properties of iterates $\{\xl_k\}$. Such conjecture, while exciting, still requires human-readable proof. Moreover, our PEP numerical result only studies one AGD parameter setting described previously in Corollary \ref{thm:AGD_par1}. It remains to derive a theoretical proof for the convergence properties of iterates $\{\xl_k\}$ for more parameter settings. 

\section{Main result}
\label{sec: convergence of constrained case}
In this section, we develop a generic theoretical result for the convergence properties of the sequence $\{\xl_k\}$ in AGD with general feasible set $X$ and include the non-Euclidean setting with general norm. Our theoretical result is motivated by the PEP numerical results in the previous section. However, unlike the numerical results that focused on a specific set of parameters, our theoretical results in this section covers a broader range of parameter and norm settings. The core idea is to bound the error term $\Delta(x)$ in Proposition \ref{thm:preprocessing} in a different way. We will assume the Euclidean setting with norm {$\|\cdot\|_2$} first, since the treatment for non-Euclidean setting will be slightly different.
Our results in the Euclidean setting are stated in the following theorem.

\begin{theorem}
\label{thm:main}
    Assume in Algorithm \ref{alg:AGD} that we use the Euclidean setting with norm {$\|\cdot\|_2$}, so that the Bregman divergence becomes $V(x_{k-1},x) = \tfrac{1}{2}{\|x - x_{k-1}\|_2^2}$. Suppose that the parameters satisfy
\begin{align}
    \label{eq:parameters}
    \gamma_1 = 1, \gamma_k \in (0,1),\ \forall k\ge 2, \text{ and }\eta_k\ge L\gamma_k,\ \forall k\ge 1.
\end{align}
Then for any $N\ge 2$, we have the following convergence results:
\begin{enumerate}
    \item If in addition \(\gamma_k\eta_k/\Gamma_k\) is non-increasing, then for all $x\in X$,
    \begin{align}
	    \label{eq:xl_convergence}
        \begin{aligned}
            f(\xl_N) - f(x) \le & \frac{\Gamma_N\eta_1}{2}{\|x_0 - x\|_2^2} 
            \\
            & + \max\left\{0,  \frac{\gamma_N^2}{\Gamma_N^2} - \frac{\gamma_{N-1}^2}{\Gamma_{N-1}^2}\right\}\frac{\Gamma_N\Gamma_{N-1}\eta_{N-1}}{2\gamma_{N-1}}{\|x_{N-1} - x\|_2^2}.
        \end{aligned}
    \end{align}
    Specifically, letting \(x = x^*\) in the above result, we have
    \begin{align}
    \label{eq:xl_convergence_de}
    f(\xl_N) - f(x^*) 
    \le & \max\left\{1,  \frac{\Gamma_{N-1}^2\gamma_N^2}{\gamma_{N-1}^2\Gamma_N^2}\right\}\frac{\Gamma_N\eta_1}{2}{\|x_0 - x^*\|_2^2}
    .
    \end{align}
    \item If in addition \(\gamma_k\eta_k/\Gamma_k\) is non-decreasing and the set $X$ is bounded with diameter $D_X:=\max_{x,y\in X}{\|x - y\|_2}$, then for all $x\in X$,
    \begin{align}
    \label{eq:xl_convergence_in}
    f(\xl_N) - f(x) 
    \le & \max\left\{1,  \frac{\Gamma_{N-1}^2\gamma_N^2}{\gamma_{N-1}^2\Gamma_N^2}\right\}\frac{\Gamma_N\gamma_{N-1}\eta_{N-1}}{2\Gamma_{N-1}}D_X^2.
    \end{align}
\end{enumerate}

\end{theorem}

\begin{proof}
    It suffices to bound the error term $\Delta(x)$ in the result of Proposition \ref{thm:preprocessing}. We make a few observations on the entries of the error term. 
First, noticing the optimality condition of \eqref{eq:x}, we have
\[\langle g_k + \eta_k(x_k-x_{k-1}), \, x_{k} - x \rangle \leq 0.\]
Thus
\begin{align*}
    & \frac{\gamma_N}{\Gamma_N}\langle g_N,\, x_{N-1}-x_N\rangle + \sum_{k=1}^N\frac{\gamma_k}{\Gamma_k}\langle g_k,\, x_{k}-x\rangle\\
    = & \frac{\gamma_N}{\Gamma_N}\langle g_N - g_{N-1},\, x_{N-1}-x\rangle + \frac{\gamma_N}{\Gamma_N}\langle g_{N-1},\, x_{N-1}-x\rangle + \sum_{k=1}^{N-1}\frac{\gamma_k}{\Gamma_k}\langle g_k,\, x_{k}-x\rangle\\
    \le & \frac{\gamma_N}{\Gamma_N}\langle g_N - g_{N-1},\, x_{N-1}-x\rangle - \frac{\gamma_N\eta_{N-1}}{\Gamma_N}\langle x_{N-1} - x_{N-2},\, x_{N-1}-x\rangle 
    \\
    & - \sum_{k=1}^{N-1}\frac{\gamma_k\eta_k}{\Gamma_k}\langle x_k - x_{k-1} ,\, x_{k}-x\rangle.
\end{align*}

Second, applying Cauchy-Schwarz and Young's inequality, we have
\begin{align*}
    & \sum_{k=2}^{N} \frac{\gamma_{k-1}}{\Gamma_{k-1}} 
\langle g_k - g_{k-1}, \, x_{k-1} - x_{k-2} \rangle  
\\
 \le &
 \frac{\gamma_{N-1}}{\Gamma_{N-1}}  \langle g_N - g_{N-1}, \, x_{N-1} - x_{N-2} \rangle + \sum_{k=2}^{N-1} \left[\frac{1}{2L\Gamma_{k-1}} {\|g_k - g_{k-1}\|_2^2} + \frac{L\gamma_{k-1}^2}{2\Gamma_{k-1}}{\|x_{k-1} - x_{k-2}\|_2^2}\right]
 \\
 \le &
 \frac{\gamma_{N-1}}{\Gamma_{N-1}}  \langle g_N - g_{N-1}, \, x_{N-1} - x_{N-2} \rangle + \sum_{k=2}^{N-1} \frac{1}{2L\Gamma_{k-1}} {\|g_k - g_{k-1}\|_2^2} + \sum_{k=1}^{N-2}\frac{\gamma_k\eta_k}{2\Gamma_k}{\|x_k - x_{k-1}\|_2^2}.
\end{align*}
Here in the last inequality we apply the assumption that $\eta_k\ge L\gamma_k$ in \eqref{eq:parameters} and changed the index in the second summand.
Third, for the terms in the first and second observations above we can further observe that
\begin{align}
	\label{eq:eucl_only}
	\begin{aligned}
     & \frac{\gamma_N}{\Gamma_N} \langle g_N - g_{N-1}, \, x_{N-1} - x \rangle + \frac{\gamma_{N-1}}{\Gamma_{N-1}}  \langle g_N - g_{N-1}, \, x_{N-1} - x_{N-2} \rangle
     \\
     = \,& \frac{\gamma_{N-1}}{2\Gamma_{N-1}\eta_{N-1}} {\|g_N - g_{N-1}\|_2^2} + \frac{\gamma_{N-1}\eta_{N-1}}{2\Gamma_{N-1}} {\|x_{N-1} - x_{N-2}\|_2^2} \\
        & + \frac{\Gamma_{N-1}\eta_{N-1}\gamma_N^2}{2\Gamma_{N}^2\gamma_{N-1}} {\|x_{N-1} - x\|_2^2}  + \frac{\gamma_N\eta_{N-1}}{\Gamma_N}\langle x_{N-1} - x_{N-2},\, x_{N-1}-x\rangle
        \\
        & - \frac{\gamma_{N-1}}{2\Gamma_{N-1}\eta_{N-1}}{\left\|(g_N - g_{N-1}) - \frac{\Gamma_{N-1}\eta_{N-1}\gamma_N}{\Gamma_{N}\gamma_{N-1}}(x_{N-1} - x) - \eta_{N-1}(x_{N-1} - x_{N-2})\right\|_2^2}.
    \end{aligned}
\end{align}
Combining all the above observations about $\Delta(x)$, dropping some negative terms, and again recalling that $\eta_{N-1} \ge L\gamma_{N-1}$ in assumption \eqref{eq:parameters}, we have
\begin{align*}
    & \Delta(x) 
    \\
    \le& - \frac{\eta_1}{2}{\|x - x_0\|_2^2} + \frac{\Gamma_{N-1}\eta_{N-1}\gamma_N^2}{2\Gamma_{N}^2\gamma_{N-1}} {\|x_{N-1} - x\|_2^2} + 
    \sum_{k=1}^{N-1}\frac{\gamma_k\eta_k}{2\Gamma_k}{\|x_k -x_{k-1}\|_2^2}
    \\
    & - \sum_{k=1}^{N-1}\frac{\gamma_k\eta_k}{\Gamma_k}\langle x_k - x_{k-1} ,\, x_{k}-x\rangle 
    \\
    = &\sum_{k=1}^{N-2}(\frac{\gamma_{k+1}\eta_{k+1}}{2\Gamma_{k+1}} - \frac{\gamma_k\eta_k}{2\Gamma_k}){\|x - x_k\|_2^2} + \frac{\Gamma_{N-1}\eta_{N-1}}{2\gamma_{N-1}}\left( \frac{\gamma_N^2}{\Gamma_N^2} - \frac{\gamma_{N-1}^2}{\Gamma_{N-1}^2}\right){\|x-x_{N-1}\|_2^2} 
    \\
    \le & \sum_{k=1}^{N-2}(\frac{\gamma_{k+1}\eta_{k+1}}{2\Gamma_{k+1}} - \frac{\gamma_k\eta_k}{2\Gamma_k}){\|x - x_k\|_2^2} 
    + \frac{\Gamma_{N-1}\eta_{N-1}}{2\gamma_{N-1}}\max\left\{0,  \frac{\gamma_N^2}{\Gamma_N^2} - \frac{\gamma_{N-1}^2}{\Gamma_{N-1}^2}\right\}{\|x-x_{N-1}\|_2^2},
\end{align*}
where we used the identity 
\[
{\|x_k-x_{k-1}\|_2^2}
- 2\langle x_k-x_{k-1},\, x_k-x\rangle
= {\|x-x_{k-1}\|_2^2}
- {\|x-x_k\|_2^2},
\]
then summed and reindexed the resulting telescoping terms.

Now we distinguish two cases. 
First, if \(\gamma_k\eta_k/\Gamma_k\) is non-increasing, then
\[
    \Delta(x) \le \,\max\left\{0,  \frac{\gamma_N^2}{\Gamma_N^2} - \frac{\gamma_{N-1}^2}{\Gamma_{N-1}^2}\right\}\frac{\Gamma_{N-1}\eta_{N-1}}{2\gamma_{N-1}}{\|x-x_{N-1}\|_2^2} .
\]
Thus the inequality \eqref{eq:xl_convergence} follows immediately from the conclusion in Proposition~\ref{thm:preprocessing}. Moreover, observe that our parameter setting described in \eqref{eq:parameters} satisfies the requirement of Theorem \ref{thm:AGD_xu_convergence}, hence equation \eqref{eq:xu_convergence_key} holds. 
Substituting our parameter setting and letting $x=x^*$ 
and $k=N-1$ in equation \eqref{eq:xu_convergence_key}, we have 
\begin{align*}
    \frac{\gamma_{N-1}\eta_{N-1}}{2\Gamma_{N-1}}{\|x_{N-1} -x^*\|_2^2} \le \frac{\eta_1}{2}{\|x_0 - x^*\|_2^2} - \frac{1}{\Gamma_{N-1}}(f(\xu_{N-1}) - f(x^*)) \le \frac{\eta_1}{2}{\|x_0 - x^*\|_2^2}.
\end{align*}
Applying the above inequality to \eqref{eq:xl_convergence}, the result in \eqref{eq:xl_convergence_de} follows directly.
Second, if \(\gamma_k\eta_k/\Gamma_k\) is non-decreasing and the set $X$ is bounded with diameter $D_X$ (i.e., $\max_{x,y\in X}{\|x - y\|_2} \le D_X)$, then

\begin{align*}
    \Delta(x) &\le \sum_{k=1}^{N-2}\left(\frac{\gamma_{k+1}\eta_{k+1}}{2\Gamma_{k+1}} - \frac{\gamma_k\eta_k}{2\Gamma_k}\right)D_X^2 + \frac{\Gamma_{N-1}\eta_{N-1}}{2\gamma_{N-1}}\max\left\{0,  \frac{\gamma_N^2}{\Gamma_N^2} - \frac{\gamma_{N-1}^2}{\Gamma_{N-1}^2}\right\}D_X^2\\
    &= \left(\frac{\gamma_{N-1}\eta_{N-1}}{2\Gamma_{N-1}}-\frac{\eta_1}{2} + \frac{\Gamma_{N-1}\eta_{N-1}}{2\gamma_{N-1}}\max\left\{0,  \frac{\gamma_N^2}{\Gamma_N^2} - \frac{\gamma_{N-1}^2}{\Gamma_{N-1}^2}\right\}\right)D_X^2.
\end{align*}
The result \eqref{eq:xl_convergence_in} follows directly. 

\end{proof}

We can apply the results in Theorem~\ref{thm:main} to several standard parameter choices (see previously in Corollaries~\ref{thm:AGD_par1}--\ref{thm:AGD_par3}) and obtain parameter-specific convergence properties of the gradient-evaluation sequence $\{\xl_k\}$. Corollaries~\ref{cor:xl_par1}--\ref{cor:xl_par3} below are the gradient-evaluation sequence counterparts of Corollaries~\ref{thm:AGD_par1}--\ref{thm:AGD_par3}. The proofs are direct applications of  Theorem~\ref{thm:main}  and hence are skipped.

\begin{corollary}
    \label{cor:xl_par1}
    If the parameters of Algorithm~\ref{alg:AGD} are set to
    \begin{align}
        \gamma_k = \frac{2}{k+1}, \qquad \eta_k = \frac{2L}{k},
    \end{align}
    then we have
    \begin{align*}
        f(\xl_N) - f(x^*) \le \frac{2NL}{(N-1)^2(N+1)}{\|x_0-x^*\|_2^2},\ \forall N\ge 2.
    \end{align*}
\end{corollary}

\begin{corollary}
    \label{cor:xl_par2}
    If the parameters of Algorithm~\ref{alg:AGD} are set to
    \begin{align}
        \gamma_k = \begin{cases}
            1 & k= 1
            \\
            \text{positive solution $\gamma$ to equation $\gamma^2 = \gamma_{k-1}^2(1-\gamma)$} & k>1,
        \end{cases}
        \qquad \eta_k = L\gamma_k,
    \end{align}
    then we have
    \begin{align*}
        f(\xl_N) - f(x^*) \le \frac{2L}{N^2}{\|x_0-x^*\|_2^2},\ \forall N\ge 2.
    \end{align*}
\end{corollary}

\begin{corollary}
    \label{cor:xl_par3}
    If the feasible set $X$ is compact with diameter $D_X:=\max_{x,y\in X}{\|x-y\|_2}$ and the parameters of Algorithm~\ref{alg:AGD} are set to
    \begin{align}
        \gamma_k = \frac{3}{k+2}, \qquad \eta_k = \frac{3L}{k},
    \end{align}
    then we have
    \begin{align*}
        f(\xl_N) - f(x) \le \frac{9L(N+1)}{2(N-1)^2(N+2)}D_X^2,\ \forall x\in X,\ \forall N\ge 2.
    \end{align*}
\end{corollary}

We will now move on to generalize the convergence properties to non-Euclidean settings, which requires a slightly different treatment than the proof of Theorem \ref{thm:main} since the relationship between inner products and Euclidean norms in \eqref{eq:eucl_only} is no longer valid. 

\begin{theorem}
\label{thm:brg_case}
    Suppose that the parameters of Algorithm \ref{alg:AGD} satisfy \eqref{eq:parameters}, i.e.,
\begin{align*}
    \gamma_1 = 1, \gamma_k \in (0,1), \forall k\ge 2, \text{ and }\eta_k\ge L\gamma_k,\ \forall k\ge 1.
\end{align*}
Then for any $N\ge 3$, we have the following convergence results:
\begin{enumerate}
    \item If in addition \(\gamma_k\eta_k/\Gamma_k\) is non-increasing, then for all $x\in X$,
    \begin{align}
    \label{eq:xl_convergence_non_eucl}
        f(\xl_N) - f(x) \le &\Gamma_N\eta_1V(x_0,x) + \frac{\Gamma_N\Gamma_{N-1}\eta_{N-1}}{\gamma_{N-1}}\max\left\{0,  \frac{\gamma_N^2}{\Gamma_N^2} - \frac{\gamma_{N-1}^2}{\Gamma_{N-1}^2}\right\}V(x_{N-1},x)\\
    & + \gamma_N\eta_{N-1} V(x_{N-2},x).
    \end{align}
    Specifically, letting \(x = x^*\) in the above result, we have
    \begin{align}
    \label{eq:xl_convergence_de_non_eucl}
    f(\xl_N) - f(x^*) 
    \le & \left(\max\left\{1,\frac{\Gamma_{N-1}^2\gamma_N^2}{\Gamma_N^2\gamma_{N-1}^2}\right\}+ \frac{\Gamma_{N-2}\gamma_N\eta_{N-1}}{\Gamma_N\gamma_{N-2}\eta_{N-2}}\right)\Gamma_N\eta_1 V(x_0,x^*).
    \end{align}
   
    \item If in addition \(\gamma_k\eta_k/\Gamma_k\) is non-decreasing, the set $X$ is bounded and $D_X:=\sqrt{\max_{x\in X^o,y\in X}V(x,y)}<\infty$, then for all $x\in X$,
    \begin{align}
    \label{eq:xl_convergence_in_non_eucl}
    f(\xl_N) - f(x) 
    \le & {\gamma_N\eta_{N-1}\left( \max\left\{\frac{\Gamma_N\gamma_{N-1}}{\Gamma_{N-1}\gamma_N},  \frac{\Gamma_{N-1}\gamma_N}{\Gamma_N\gamma_{N-1}} \right\} + 1\right)D_X^2}.
    \end{align}
\end{enumerate}
\end{theorem}

\begin{proof}
    First, by Proposition \ref{thm:preprocessing}, we have
\begin{align}
    \begin{aligned}
        f(\xl_N) - f(x) \le &\, \Gamma_N\left[\frac{\gamma_N}{\Gamma_N}\langle g_N, x_{N-1} - x_N \rangle \right.\\
        &+ \sum_{k=2}^N \frac{\gamma_{k-1}}{\Gamma_{k-1}}
        \langle g_k - g_{k-1}, x_{k-1} - x_{k-2} \rangle
         - \sum_{k=2}^N \frac{1}{2L\Gamma_{k-1}}\|g_{k-1} - g_k\|_*^2 \\
        &- \left.\sum_{k=1}^N \frac{\gamma_k}{2L\Gamma_k}\|g - g_k\|_*^2
         + \sum_{k=1}^N \frac{\gamma_k}{\Gamma_k}\langle g_k, x_k - x \rangle \right].
    \end{aligned}
\end{align}
    Second, in the above relation we have
    \begin{align*}
    & \frac{\gamma_N}{\Gamma_N}\langle g_N,\, x_{N-1}-x_N\rangle + \sum_{k=1}^N\frac{\gamma_k}{\Gamma_k}\langle g_k,\, x_{k}-x\rangle\\
    = & \frac{\gamma_N}{\Gamma_N}\langle g_N - g_{N-1},\, x_{N-1}-x\rangle + \frac{\gamma_N}{\Gamma_N}\langle g_{N-1},\, x_{N-1}-x\rangle + \sum_{k=1}^{N-1}\frac{\gamma_k}{\Gamma_k}\langle g_k,\, x_{k}-x\rangle\\
    \le & \frac{\gamma_N}{\Gamma_N}\langle g_N - g_{N-1},\, x_{N-1}-x\rangle + \frac{\gamma_N\eta_{N-1}}{\Gamma_N} (V(x_{N-2},x)- V(x_{N-1},x) - V(x_{N-1},x_{N-2}))
    \\
    & + \sum_{k=1}^{N-1}\frac{\gamma_k\eta_k}{\Gamma_k}(V(x_{k-1},x) - V(x_k,x) - V(x_{k-1},x_k)).
    \end{align*}
    Here in the inequality we use the optimality condition of subproblem \eqref{eq:x} with Bregman divergence, which is
$
       \langle g_k, x_k - x \rangle \le {\eta_k}(V(x_{k-1},x) - V(x_k,x) - V(x_{k-1},x_k)).
$
    For a reference on such optimality condition with proof, see in, e.g., Lemma 3.4 in \cite{lan2020first}.
    Third, applying Cauchy-Schwarz inequality and Young's inequality, and using the condition $\eta_k\ge L\gamma_k$, we have, 
\begin{align*}
  & \sum_{k=2}^{N} \frac{\gamma_{k-1}}{\Gamma_{k-1}} 
  \langle g_k - g_{k-1}, \, x_{k-1} - x_{k-2} \rangle  \\
 \le & \frac{\gamma_{N-1}}{\Gamma_{N-1}}  \langle g_N - g_{N-1}, \, x_{N-1} - x_{N-2} \rangle + \sum_{k=2}^{N-1} \frac{1}{2L\Gamma_{k-1}} \|g_k - g_{k-1}\|_*^2 + \frac{L\gamma_{k-1}^2}{2\Gamma_{k-1}}\|x_{k-1} - x_{k-2}\|^2
 \\
 \le & \frac{\gamma_{N-1}}{\Gamma_{N-1}}  \langle g_N - g_{N-1}, \, x_{N-1} - x_{N-2} \rangle + \sum_{k=2}^{N-1} \frac{1}{2L\Gamma_{k-1}} \|g_k - g_{k-1}\|_*^2 + \sum_{k=1}^{N-2}\frac{\gamma_k\eta_k}{2\Gamma_k}\|x_k - x_{k-1}\|^2.
\end{align*}
Fourth, by using Cauchy-Schwarz and Young's inequalities and the triangle inequality for norms, we obtain
\begin{align*}
     & \frac{\gamma_N}{\Gamma_N} \langle g_N - g_{N-1}, \, x_{N-1} - x \rangle + \frac{\gamma_{N-1}}{\Gamma_{N-1}}  \langle g_N - g_{N-1}, \, x_{N-1} - x_{N-2} \rangle
     \\
     =\,& \frac{\gamma_{N-1}}{\Gamma_{N-1}}\langle g_N-g_{N-1}, \frac{\Gamma_{N-1}\gamma_N}{\gamma_{N-1}\Gamma_N}(x_{N-1}-x) + (x_{N-1}-x_{N-2})\rangle
     \\
     \le \,& \frac{\gamma_{N-1}}{2\Gamma_{N-1}\eta_{N-1}} \|g_N - g_{N-1}\|_*^2 +
    \frac{\gamma_{N-1}\eta_{N-1}}{2\Gamma_{N-1}}\left\|\frac{\Gamma_{N-1}\gamma_N}{\Gamma_{N}\gamma_{N-1}}(x_{N-1} - x) +(x_{N-1} - x_{N-2})\right\|^2
    \\
    \le \,& \frac{\gamma_{N-1}}{2\Gamma_{N-1}\eta_{N-1}} \|g_N - g_{N-1}\|_*^2 +
    \frac{\gamma_{N-1}\eta_{N-1}}{2\Gamma_{N-1}}\left[\frac{\Gamma_{N-1}\gamma_N}{\Gamma_N\gamma_{N-1}}\|x_{N-1}-x\| + \|x_{N-1}-x_{N-2}\|\right]^2
    \\
    \le \, &\frac{\gamma_{N-1}}{2\Gamma_{N-1}\eta_{N-1}} \|g_N - g_{N-1}\|_*^2 +
    \frac{\Gamma_{N-1}\gamma_N^2\eta_{N-1}}{2\Gamma_{N}^2\gamma_{N-1}}\|x_{N-1} - x\|^2\\
    &+\frac{\gamma_{N-1}\eta_{N-1}}{2\Gamma_{N-1}}\|x_{N-1} - x_{N-2}\|^2 + \frac{\gamma_N\eta_{N-1}}{\Gamma_N}\|x_{N-1}-x\|\|x_{N-1} - x_{N-2}\|.
\end{align*}
Combining all the observations above, dropping some negative terms, and using the condition $\eta_{N-1}\ge L\gamma_{N-1}$, we have 
\begin{align*}
    & \frac{1}{\Gamma_N}(f(\xl_N)-f(x) )
    \\
    \le\,&\eta_1V(x_0,x) + \sum_{k=1}^{N-2}(\frac{\gamma_{k+1}\eta_{k+1}}{\Gamma_{k+1}} - \frac{\gamma_k\eta_k}{\Gamma_k})V(x_k,x) 
    \\
    & + \sum_{k=1}^{N-1}\frac{\gamma_k\eta_k}{\Gamma_k}(\frac{1}{2}\|x_k - x_{k-1}\|^2 - V(x_{k-1},x_k)) + \frac{\Gamma_{N-1}\gamma_N^2\eta_{N-1}}{2\Gamma_{N}^2\gamma_{N-1}}\|x_{N-1} - x\|^2\\
    & + \frac{\gamma_N\eta_{N-1}}{\Gamma_N} \left(V(x_{N-2},x) - V(x_{N-1},x) - V(x_{N-1},x_{N-2}) + \|x_{N-1}-x\|\|x_{N-1} - x_{N-2}\|\right)
    \\
    & -  \frac{\gamma_{N-1}\eta_{N-1}}{\Gamma_{N-1}}V(x_{N-1},x)
    \\
    \le \, &\eta_1V(x_0,x) + \sum_{k=1}^{N-2}(\frac{\gamma_{k+1}\eta_{k+1}}{\Gamma_{k+1}} - \frac{\gamma_k\eta_k}{\Gamma_k})V(x_k,x) \\
    & + \frac{\Gamma_{N-1}\eta_{N-1}}{\gamma_{N-1}}\left( \frac{\gamma_N^2}{\Gamma_N^2} - \frac{\gamma_{N-1}^2}{\Gamma_{N-1}^2}\right)V(x_{N-1},x)+ \frac{\gamma_N\eta_{N-1}}{\Gamma_N} V(x_{N-2},x).
\end{align*}

Here in the second inequality, we use the property \(\frac{1}{2}\|x-y\|^2 \le V(x,y)\) of Bregman divergence in \eqref{eq:Vbound} and
\begin{align*}
    & \|x_{N-1}-x\|\|x_{N-1} - x_{N-2}\| - V(x_{N-1},x) - V(x_{N-1},x_{N-2})
    \\
    \le & \|x_{N-1}-x\|\|x_{N-1} - x_{N-2}\| - \frac{1}{2}\|x_{N-1}-x\|^2 - \frac{1}{2}\|x_{N-1}-x_{N-2}\|^2 \le 0.
\end{align*}
Then the rest of proof is similar to Theorem \ref{thm:main} and is skipped.
\end{proof}

Similarly, the following corollaries are parameter-specific convergence properties under the non-Euclidean case. The proofs are skipped since they are direct applications of Theorem \ref{thm:brg_case}.

\begin{corollary}
    \label{cor:xl_par4}
    If the parameters of Algorithm~\ref{alg:AGD} are set to
    \begin{align}
        \gamma_k = \frac{2}{k+1}, \qquad \eta_k = \frac{2L}{k},
    \end{align}
    then we have
    \begin{align*}
        f(\xl_N) - f(x^*) \le \frac{4(2N-1)L}{(N-1)^2(N+1)}V(x_0,x^*),\ \forall N\ge 3.
    \end{align*}
\end{corollary}

\begin{corollary}
    \label{cor:xl_par5}
    If the parameters of Algorithm~\ref{alg:AGD} are set to
    \begin{align}
        \gamma_k = \begin{cases}
            1 & k= 1
            \\
            \text{positive solution $\gamma$ to equation $\gamma^2 = \gamma_{k-1}^2(1-\gamma)$} & k>1,
        \end{cases}
        \qquad \eta_k = L\gamma_k,
    \end{align}
    then we have
    \begin{align*}
        f(\xl_N) - f(x^*) \le \frac{4(2N+1)L}{N^2(N+1)}V(x_0,x^*),\ \forall N\ge 3.
    \end{align*}
\end{corollary}

\begin{corollary}
    \label{cor:xl_par6}
    If the feasible set $X$ is compact with  $D_X:=\sqrt{\max_{x\in X^o,y\in X}V(x,y)}<\infty$ and the parameters of Algorithm~\ref{alg:AGD} are set to
    \begin{align}
        \gamma_k = \frac{3}{k+2}, \qquad \eta_k = \frac{3L}{k},
    \end{align}
    then we have
    \begin{align*}
        f(\xl_N) - f(x) \le \frac{18NL}{(N-1)^2(N+2)}D_X^2,\ \forall x\in X,\ \forall N\ge 3.
    \end{align*}
\end{corollary}

To the best of our knowledge, the results in Corollaries \ref{cor:xl_par4}--\ref{cor:xl_par6} concerning gradient-evaluation sequences in AGD have never been developed in the literature.

\section{Concluding remarks}
\label{sec:concluding_remarks}

This paper studies the gradient-evaluation sequence in accelerated gradient methods with projection and shows that, under standard AGD parameter settings, the sequence $\{\xl_k\}$ enjoys an $\mathcal{O}(1/k^2)$ function-value rate comparable to the classical rate for $\{\xu_k\}$. Our results are valid for general constrained feasible sets under possibly non-Euclidean settings.
We hope that our discovery will contribute to the broader line of research on the mechanisms underlying acceleration by deepening our knowledge on the gradient-evaluation sequence.
Our analysis is motivated by PEP-based numerical evidence but yields a complete and human-readable proof that applies to a broad range of settings, including both non-increasing and non-decreasing $\{\gamma_k\eta_k/\Gamma_k\}$ parameter regimes and non-Euclidean settings.

We emphasize that our intent is not to optimize the constant in the $\mathcal{O}(1/k^2)$ bound via parameter tuning; this goal has already been achieved by OGM, and the resulting method no longer matches the classical AGD structure. Our focus is on understanding AGD itself and the convergence behavior of its gradient-evaluation sequence. It might be of some interest to study AGD parameter settings that achieve the $\mathcal{O}(1/k^2)$ convergence rate with smaller universal constants, but such effort is out of scope of this paper and we leave it as a future research direction.
One other potential future research direction is a systematic PEP-to-proof workflow for first-order algorithms, which could further streamline the discovery of convergence properties.

\section{Acknowledgement}
This work is partially supported by AFOSR grant FA9550-22-1-0447.

\bibliography{yuyuan}

\end{document}